\documentclass{amsart}

\usepackage{amsmath,bm}
\usepackage{amssymb}
\usepackage{amsthm}
\usepackage{mathtools}

\usepackage[abbrev]{amsrefs}
\usepackage{url}

\newtheorem{theorem}{Theorem}[section]
\newtheorem{lemma}[theorem]{Lemma}

\theoremstyle{definition}

\theoremstyle{remark}
\newtheorem{remark}[theorem]{Remark}

\theoremstyle{cor}
\newtheorem{cor}[theorem]{Corollary}

\theoremstyle{prop}

\numberwithin{equation}{section}

\newcommand{\ob}[1]{\mkern 1.5mu\overline{\mkern-1.5mu#1\mkern-1.5mu}\mkern 1.5mu}
\newcommand{\N}{\mathbb{N}}

\newcommand{\R}{\mathbb{R}}
\newcommand{\eps}{\varepsilon}
\newcommand{\dl}{\delta}

\newcommand{\bone}{\mathbf{1}}
\newcommand{\wg}{\wedge}

\newcommand{\relmiddle}[1]{\mathrel{}\middle#1\mathrel{}}
\DeclareMathOperator*{\esssup}{ess\,sup}

\begin{document}

\title{Transition densities of reflecting Brownian motions on Lipschitz domains}


\author{Kouhei Matsuura}
\address{Department of Mathematics, Kyoto University, Kyoto 606–8502, Japan}
\curraddr{}
\email{koumatsu@math.kyoto-u.ac.jp}
\thanks{}



\date{}

\dedicatory{}

\begin{abstract}
In this paper, we study the continuity of the transition density of the reflecting Brownian motion on a general Lipschitz domain. We also provide local estimates for the density. Applying the estimates, we prove that the surface measure on the domain is in the local Kato class of the reflecting Brownian motion.
\end{abstract}
\maketitle

\section{Introduction}
Let $D$ be a connected open subset of $\R^d$, $d \ge 2$. We denote by $\ob{D}$ the closure of $D$ in $\R^d$.  We denote by $H^{1}(D)$ the first order Sobolev space on $D$ with the Neumann boundary condition. For each $ f,g \in H^{1}(D)$, we define $\mathcal{E}(f,g)=(1/2)\int_{D}(\nabla f, \nabla g)\,dm$, where $m$ is the Lebesgue measure on $D$ and $(\cdot, \cdot)$ is the standard inner product on $\R^d$. $\nabla f$ and $\nabla g$ denote the distributional derivatives of $f$ and $g$, respectively. If the boundary of $D$ is locally expressible as a graph of a continuous function of $(d-1)$-variables, the Dirichlet form $( \mathcal{E}, H^{1}(D))$ is regular on $L^{2}(\ob{D},m)$ (see \cite[Chapter~V, Theorem~4.7]{EE} for details), and generates a diffusion process $X=(\{X_t\}_{t \ge 0},\{ P_x\}_{x \in \ob{D}})$ on $\ob{D}$. We call $X$ the {\it reflecting Brownian motion} (RBM in abbreviation) on $\ob{D}$.  We denote by $p_{t}(x,dy)$ the transition probability of $X$. Under suitable assumptions on $D$, the transition function is absolutely continuous with respect to $m$, and the density $p_{t}(x,y)$ is called the {\it heat kernel} of $X$. 

When the heat kernel $p_{t}(x,y)$ is continuous on $\ob{D} \times \ob{D}$? Bass and Hsu show in \cite[Lemma~4.3]{BH} that $p_{t}(x,y)$ is continuous if $D$ is a bounded Lipschitz domain. 
In \cite{FT}, Fukushima and Tomisaki  prove the continuity of the resolvent density for the RBM on the closure of a global Lipschitz domain with H\"{o}lder cusps. See \cite[Theorem~2.1~(iii)]{FT} for details. In \cite[Proposition~3.4]{FT}, they also prove the Sobolev type inequality: there exist positive constants $S>0$ and $p >2$ such that
\begin{equation}
\|f\|_{L^{p}(D,m)} \le S \|f\|_{H^{1}(D)} \label{eq:sobolev}
\end{equation}
for any $f \in H^{1}(D)$. Here, we denote by $\|\cdot \|_{H^{1}(D)}$ the standard norm on $H^{1}(D)$. By \eqref{eq:sobolev} and \cite[Theorem~6.10]{O}, for any $\eps>0$, there exist positive constants $a_\eps ,b_\eps \in (0,\infty)$ depending on $d$ and $D$, and $\eps>0$ such that 
\begin{equation}
p_{t}(x,y) \le a_{\eps}e^{\eps t} t^{-d/2} \exp\left(-\frac{|x-y|^2}{b_{\eps}t} \right) \label{eq:gauss}
\end{equation}
for any $t>0$ and $m \otimes m$-a.e. $(x,y) \in D \times D$. Here, we denote by $|\cdot|$ the Euclidean norm on $\R^d$. It is also known that \eqref{eq:gauss} implies \eqref{eq:sobolev}. See  \cite[Theorem~4.2.7]{FOT} for details.

Although the framework of \cite{FT} is seemingly wide, there are many domains on which \eqref{eq:sobolev} does not hold. For example, the Sobolev type inequality \eqref{eq:sobolev} does not hold on the domain $D_H$ defined as
\begin{align*}
&D_H=\{(x,y) \in \R \times \R^{d} \mid x>1,\ |y|<H(x)\},
 \end{align*}
 where $H(x)$ is a positive continuous function on $\R$ such that $\lim_{x \to \infty}H(x)=0$. An essential reason why the the Sobolev type inequality \eqref{eq:sobolev} fails on $D_H$ is a presence of a cusp at infinity.
 Hence, $D_H$ is not a Lipschitz domains in the sense of \cite{FT}. Relatively recently, Gyrya and Saloff-Coste prove in \cite[Theorem~3.10]{GS} that heat kernels of the RBMs on inner uniform domains are continuous. However, $D_H$ is not an inner uniform domain. It seems that there is no preceding results which prove the continuity of the heat kernels of RBMs on horn-shaped domains like as $D_H$.

In this paper, we obtain the continuity of the heat kernel of the RBM on a general Lipschitz domain. For the proof, it is important to show that part processes of the RBM are identified with part processes of RBMs on bounded Lipschitz domains (Lemma~\ref{lem:ind}). This kind of argument is found in \cite[Lemma~6.2]{Ma}, where the author essentially uses the theory of Sobolev extension domains. We use the theorem of the spectral synthesis, and the proof of Lemma~\ref{lem:ind} is much simpler than that of \cite[Lemma~6.2]{Ma}. Combining Lemma~\ref{lem:ind} with the result of Grigor'yan and Kajino~\cite[Theorem~1.1]{GK}, we also establish local estimates of the heat kernel of the RBM (Theorem~\ref{thm:thm1}~(2)). We apply the estimates to prove that the surface measure on the boundary of a Lipschitz domain is in the local Kato class of the RBM on it (Theorem~\ref{localkato}). To classify measures in this way is important in the transformation theory of the Markov processes. See \cite{CK} and \cite{KKT} for the transformation theory and its applications. The local estimates will also be used in \cite{Ma2} to study the $L^p$-spectral independence of Neumann Laplacians on horn-shaped domains.\\

\noindent
{\it Notation.}
\noindent
Throughout this paper, we adopt the following notation.
\begin{itemize}
\item[(1)] For a topological space $E$, we denote by $\mathcal{B}(E)$ the Borel $\sigma$-algebra on $E$. For each $p \in [1,\infty]$ and each positive Borel measure $\mu$ on $E$, we denote by $L^{p}(E,\mu)$ the $L^p$-space on $(E,\mu)$. For each $f : E \to \R$, we write $\|f\|_{E,\infty}$ for $\sup_{x \in E}|f(x)|$. We also write 
\begin{align*}
\mathcal{B}_{b}(E)&=\{f:E \to \R \mid f\text{ is Borel measurable and }\|f\|_{E,\infty}<\infty\}, \\
C(E)&=\{f:E \to \R \mid f \text{ is continuous on }E\},\\
C_{b}(E)&=C(E) \cap \mathcal{B}_{b}(E), \\
C_{c}(E)&=\{f:E \to \R \mid \text{support of }f \text{ is a compact subset of }E\}.
\end{align*}
We denote by $C_{\infty}(E)$ the completion of $C_{c}(E)$ with respect to the norm $\|\cdot\|_{E,\infty}$.
 \item[(2)]
 Let $d \ge 2$ be an integer. $B(x,R)$ denotes open ball of $\R^d$ centered at $x \in \R^d$ with radius $R>0$. If $x$ is the origin  of $\R^{d}$, we write $B(R)$ for $B(x,R)$. The $d$-dimensional Lebesgue measure is denoted by $m$ or $dx$. 
For an open subset $E \subset \R^d$, we define $H^{1}(E)$ by
\begin{align*}
H^{1}(E)&=\left\{ f \in L^{2}(E,m) \relmiddle|  \frac{\partial f}{\partial x_i}\in L^{2}(E,m), 1 \le i \le d\right\}, 
\end{align*}
where $\partial f/\partial x_i$ is the distributional derivative of $f$ on $E$. For each $f\in H^1(E)$, we set $\| f\|_{H^1(E)}^2 := \sum_{i=1}^{d}\int_{E}|\partial_i f|^{2}\,dm+\int_{E}|f|^{2}\,dm.$
\end{itemize}

\section{Main results}
Let $d \ge 2$ be an integer and $D$ a connected open subset of $\R^d$. We denote by $\ob{D}$ the closure of $D$ in $\R^d$. In what follows, we assume that $D$ is a Lipschitz domain in the sense that:
\begin{itemize}
\item[] for any compact subset $K$ of $\ob{D}$, there exists a bounded open subset $U$ of $\R^d$ such that $K \subset U$ and $D\cap U$ is a bounded Lipschitz domain of $\R^d$.
\end{itemize}
See \cite[Definition~7.1]{Ma} for the definition of bounded Lipschitz domains. For each $f,g \in H^{1}(D)$, we define $\mathcal{E}(f,g)$ by
\begin{equation*}
\mathcal{E}(f,g)=\frac{1}{2}\sum_{i=1}^{d}\int_{D}\partial_i f\partial_i g\,dm.
\end{equation*}
It is shown in \cite[Proposition~7.3]{Ma} that $(\mathcal{E},H^{1}(D))$ becomes a regular Dirichlet form on $L^{2}(\ob{D},m)$. That is, $H^{1}(D) \cap C_{c}(\ob{D})$ is a dense subspace of $(H^{1}(D), \|\cdot\|_{H^{1}(D)})$ and of $(C_{c}(\ob{D}), \|\cdot \|_{\ob{D}, \infty})$. It is shown in  \cite[Theorem~6.10]{Ma} that $(\mathcal{E},H^{1}(D))$ generates a Hunt process on $\ob{D}$ with the semigroup strong Feller property:
\begin{theorem}\label{thm:str} There exists a Hunt process $X=( \{X_t\}_{t \ge 0}, \{P_x\}_{x \in \ob{D}})$ associated with $(\mathcal{E},H^{1}(D))$ such that whose semigroup $\{p_t\}_{t>0}$ satisfies the following: for any $f \in \mathcal{B}_{b}(\ob{D})$ and $t>0$, $p_{t}f \in C_{b}(\ob{D})$. 
\end{theorem}
\begin{remark}\label{fel}
If $p_{t}(C_{\infty}(\ob{D})) \subset C_{\infty}(\ob{D})$ for any $t>0$, $X$ is called a Feller process. If $D$ is a horn-shaped domain, $X$ can be uniformly ergodic and is not always a Feller process (\cite[Proposition~2.11]{BCM}). However, reflecting Brownian motions constructed in \cite{FT} have Feller property (\cite[Theorem~2.1~(ii)]{FT}). 
\end{remark}

By Theorem~\ref{thm:str} and \cite[Exercise~4.2.4]{FOT}, the transition kernel of $X$ is absolutely continuous with respect to the Lebesgue measure $m$:
\begin{equation}
p_{t}(x,dy)=p_{t}(x,y)\,dm(y)  \quad \text{for each $t>0$ and $x \in \ob{D}$}. \label{ac}
\end{equation}

Since $(\mathcal{E}, H^{1}(D))$ is a strongly local Dirichlet form, by \cite[Theorem~4.5.3]{FOT}, $X$ is a diffusion process on $\ob{D}$. Furthermore, $X$ is conservative by Takeda's test. See \cite[Exercise~5.7.1]{FOT} for the proof. Hence, it follows that for any $x \in \ob{D}$
\begin{align*}
&P_{x}(X_t \in \ob{D} \text{ for any }t \in [0,\infty) \text{ and } [0,\infty) \ni t \mapsto X_t \in \ob{D}\text{ is continuous})=1. 
\end{align*}

For each $R>0$ and $\eps \in (0,1)$, we define
\begin{align*}
\ob{D}_R&=\ob{D} \cap B(R), \\
\ob{D}_{\eps,R}&=\{x \in \ob{D} \mid \inf_{y \in \ob{D}\setminus B(R)}|x-y|>\eps R \}.
\end{align*}
We note that each $\ob{D}_R$ and $\ob{D}_{\eps,R}$ are open subsets of $\ob{D}$. For an open subset $U \subset \ob{D}$, we define $\tau_{U}=\inf \{t>0 \mid X_t \notin U\}$ with convention that $\inf  \emptyset=\infty$. 

We are ready to state our main results. 
\begin{theorem}\label{thm:thm1}
Let $R \in (0,\infty)$ and $\eps \in (0,1)$. 
\begin{enumerate}
\item There are positive constants $c_R$, $\gamma_R$ depending on $R$ such that
\begin{align*}
P_{x}(\tau_{\ob{D} \cap B(x,r)} \le t) \le c_R \exp (-\gamma_R r^2/t)
\end{align*}
for any $t \in (0,\infty)$ and $(x,r) \in \ob{D}_R \times (0,R)$ with $B(x,r) \subset B(R)$. 
\item
There is a constant $a_{R}>0$ depending on $R$ such that for $m$-a.e. $y \in \ob{D}_{\eps, R}$,
\begin{equation*}
p_{t}(x,y) \le 
\begin{cases}
c_{R,\eps}a_R e^{t} t^{-d/2} \exp(- \eps \gamma_R |x-y|^2/t ) & \text{ if } t<R^2\text{ and } x \in \ob{D}_{R}, \\
c_{R,\eps}a_R  e^{ (2t) \wg  R^2} \{ (2t) \wg  R^2 \}^{-d/2}  \exp(-\eps \gamma_R R^2/t ) & \text{ if } t<R^2\text{ and } x \notin \ob{D}_{R}, \\
c_{R,\eps}a_R  e^{R^2}R^{-d} & \text{ if } t \ge R^2 \text{ and } x \in \ob{D}
\end{cases}
\end{equation*}
for some $c_{R,\eps}>0$ depending on $c_R, \gamma_R$, and $ \eps$.
\item $p_{t}(x,y)$ has a version which is positive and continuous on $(0,\infty) \times \ob{D} \times \ob{D}$. 
\end{enumerate}
\end{theorem}
We denote by $\ob{D}_{\partial}=\ob{D} \cup \{\partial\}$ the one-point compactification of $\ob{D}$. For an open subset  $U \subset \ob{D}$, we define $X^{U}=(\{X^{U}_t\}_{t \ge 0}, \{P_x\}_{x \in U})$ by 
\begin{equation*}
X^U_{t}=
\begin{cases}
 X_t & \text{ if }t<\tau_{U}, \\
 \partial & \text{ if } t \ge \tau_U.
\end{cases}
\end{equation*}
$X^U$ is called the {\it part process} of $X$ on $U$.
Clearly, the transition kernel $p_{t}^{U}(x,dy)$ of $X^{U}$ is absolutely continuous with respect to the Lebesgue measure $m$: 
\begin{equation}
p_{t}^{U}(x,dy)=p_{t}^{U}(x,y)\,dm(y) \quad \text{for each } t>0\text{ and } x\in U. \label{ac2}
\end{equation}

$p_{t}^{U}(x,y)$ also possesses a continuous version.
\begin{theorem}\label{cor:2}
For any non-empty open subset $U \subset \ob{D}$, $p_{t}^{U}(x,y)$ has a version which is continuous on $(0,\infty) \times U \times U$. If $U$ is connected, in addition, the version is positive.
\end{theorem}

\section{Preliminaries}
Since $D$ is a Lipschitz domain, there exist increasing bounded open subsets $\{U_n\}_{n=1}^{\infty}$ of $\R^d$ such that for each $U_n \cap D$ is a bounded Lipschitz domain of $\R^d$ and $\ob{D}=\bigcup_{n=1}^{\infty} U_n \cap \ob{D}$. For each $n \in \N$, we set 
\begin{equation*}
I_n:=D \cap U_n,\quad J_n:=\ob{I_n},\quad K_n:=\ob{D} \cap U_n.
\end{equation*}
Here, $\ob{I_n}$ is the closure of $I_n$ in $\R^d$. $K_n$ is an open subset of $\ob{D}$. For each $n \in \N$, we define $\tau_{n}=\tau_{U_n}.$ We denote by $X^{n}=(\{X_t^{n}\}_{t \ge 0}, \{P_x\}_{x \in K_n})$ the part process of $X$ on $K_n$. The semigroup is denoted by $\{p_t^n\}_{t>0}$. We set 
$$\mathcal{C}_{K_n}=\{f \in H^{1}(D) \cap C_{c}(\ob{D}) \mid \text{supp}[f] \subset K_n\}.$$
By \cite[Lemma~2.3.4~(ii)]{FOT}, the Dirichlet form $(\mathcal{E}^n,\mathcal{F}^n)$ of $X^n$ is regular on $L^{2}(K_n,m)$. It also holds that
\begin{align*}
\mathcal{F}^n&=\text{the completion of }\mathcal{C}_{K_n}\text{ with respect to }\|\cdot\|_{H^{1}(D)},\\
 \mathcal{E}^n&=\mathcal{E}|_{\mathcal{F}^n \times \mathcal{F}^n}.
\end{align*}

 Since each $I_n$ is a bounded Lipschitz domain, by \cite[Theorem~3.1]{BH}, there exists a reflecting Brownian motion $Y^n=(\{Y_t^n\}_{t \ge 0}, \{Q_{x}^{n}\}_{x \in J_n})$ on $J_n$ with the following properties (see also \cite[Theorem~3.10]{GS}).
\begin{itemize}
\item The Dirichlet form $(\mathcal{A}^n,\mathcal{B}^n)$ of $Y^n$ is identified with
\begin{align*}
\mathcal{B}^n=H^{1}(I_n),\quad \mathcal{A}^n (f,g)=\frac{1}{2}\sum_{i=1}^{d}\int_{I_n}\partial_ i f \partial_i g\,dm,\quad f,g \in \mathcal{B}^n.
\end{align*}
\item The semigroup $\{q_t^n\}_{t>0}$ of $Y^n$ satisfies the following: for any $t>0$ and any $f \in \mathcal{B}_{b}(J_n)$, $q_{t}^nf$ is a bounded continuous function on $J_n$.
\item The transition kernel $q_{t}^{n}(x,dy)$ of $Y^n$ is absolutely continuous with respect to $m$ and the density $q_{t}^{n}(x,y)$ is continuous on $(0,\infty) \times J_n \times J_n$. There exist constants $a_n, b_n \in (0,\infty)$ such that
\begin{equation}
q_{t}^{n}(x,y) \le a_n e^{t}t^{-d/2}\exp \left(-|x-y|^2/b_n t \right) \label{nhke}
\end{equation}
for any $t \in (0,\infty)$ and $x,y \in J_n$.
\end{itemize}
For each $n \in \N$, $K_n$ is also an open subset of $J_{n+1}$. We denote $Y^{n+1,n}$ by the part process of $Y^{n+1}$ on $K_n$. It follows from \cite[Theorem~1]{CK} that the semigroup of $Y^{n+1,n}$ is strong Feller: for any $f \in \mathcal{B}_{b}(K_n)$ and $t>0$, $q_{t}^{n+1,n}f$ is bounded continuous on $K_n$. The Dirichlet form  $(\mathcal{A}^{n+1},\mathcal{B}^{n+1})$ is regular on $L^{2}(J_{n+1},m)$. Hence, by \cite[Lemma~2.3.4~(ii)]{FOT},  the Dirichlet form $(\mathcal{A}^{n+1,n},\mathcal{B}^{n+1,n})$ of $Y^{n+1,n}$ is regular on $L^{2}(K_n,m)$. It also holds that 
\begin{align*}
\mathcal{B}^{n+1,n}&=\text{the completion of }\mathcal{C}'_{K_n}\text{ with respect to }\|\cdot\|_{H^{1}(I_{n+1})},\\
 \mathcal{A}^{n+1,n}&=\mathcal{A}^{n+1}|_{\mathcal{B}^{n+1,n} \times \mathcal{B}^{n+1,n}},
\end{align*}
where $\mathcal{C}'_{K_n}=\{f \in H^{1}(I_{n+1}) \cap C_{c}(J_{n+1}) \mid \text{supp}[f] \subset K_n\}$.

There is an indirect relation between $Y^n$ and $X$. Identifying the Dirichlet forms of $Y^{n+1,n}$ and $X^n$, we obtain the following lemma.
\begin{lemma}\label{lem:ind}
It holds that
$$p_{t}^{n}f(x)=q_{t}^{n+1,n}f(x),\quad x \in K_n$$
for any $n \in \N$, $t>0$, and $f \in \mathcal{B}_{b}(K_n)$.
In particular, $X^n$ is strong Feller: $p_t^{n}f$ is continuous on $K_n$.
\end{lemma}
\begin{proof}
The Dirichlet forms of $X^{n}$ and $Y^{n+1,n}$ conincide. Indeed, $\mathcal{C}_{K_n}$ and $\mathcal{C}'_{K_n}$ coincide as the subspace of $L^{2}(K_n,m)$, and the norms $\|\cdot\|_{H^{1}(D)}$ and $\|\cdot\|_{H^{1}(I_{n+1})}$ are equivalent on $\mathcal{C}_{K_n}$.  Since the Dirichlet forms coincide, it holds that $p_{t}^{n}f=q_{t}^{n+1,n}f$, $m$-a.e.  for any $t>0$ and $f \in \mathcal{C}_{b}(K_n)$.  It follows from \eqref{ac2} that for any $\eps>0$ and $x \in K_n$, 
\begin{equation}
p_{t+\eps}^{n}f(x)=p_{\eps}^{n}(p_t^{n}f)(x)=p_{\eps}^{n}(q_{t}^{n+1,n}f)(x).\label{eq:eqapp}
\end{equation}
$q_{t}^{n+1,n}f$ is continuous on $K_n$. Therefore, by letting $\eps \to 0$ in \eqref{eq:eqapp}, we have $p_{t}^{n}f(x)=q_{t}^{n+1,n}f(x)$ for any $x \in K_n$. A monotone class argument completes the proof.
\end{proof}

\begin{remark}
\begin{itemize}
\item[(i)] The proof of Lemma~\ref{lem:ind} is much simpler than that of \cite[Lemma~6.2]{Ma}, where the author uses the theory of extension domains.
\item[(ii)]
If $X$ is a Feller process, we can apply \cite[Theorem~1]{CK} to $X$ and obtain the strong Feller property of $X^n$. 
\end{itemize}
\end{remark}

It follows from  Lemma~\ref{lem:ind} and \eqref{nhke} that each $p_t^n$ is a bounded operator from $L^{1}(K_n,m)$ to $L^{\infty}(K_n,m)$. In particular, each $p_t^n$ becomes a compact operator on $L^{2}(K_n,m)$. Therefore, the (non-positive) generator $\mathcal{L}^n$ of $\{p_t^n\}_{t>0}$ has no essential spectrum.
\begin{lemma}\label{lem:contiversion}
For each $n \in \N$, the eigenfunctions of $-\mathcal{L}^n$ has a bounded continuous version on $K_n$. The principal eigenfunction can be taken to be positive on $K_n$.
\end{lemma}
\begin{proof}
We denote by $\{\lambda_k\}_{k=1}^{\infty} \subset [0,\infty)$ the eigenvalues of $-\mathcal{L}^n$. Then, the eigenfunctions $\{\varphi_k\}_{n=1}^{\infty}$ of $-\mathcal{L}^n$ satisfy $-\mathcal{L}^n \varphi_k=\lambda_k \varphi_k$ for each $k \in \N$. It is easy to see that $\varphi_{k}=e^{-\lambda_{k}}p_{1}^n\varphi_{k}$ and it follows that each $\varphi_{k}$ has a bounded continuous version by Lemma~\ref{lem:ind} and the ultracontractivity of $\{p_t^n\}_{t>0}$. Since $K_n$ is connected, $X^n$ is irreducible in the sense of \cite[Section~1]{FOT}. Since $\{p_t^n\}_{t>0}$ is a bounded operator from $L^{1}(K_n,m)$ to $L^{\infty}(K_n,m)$, $X^n$ possesses a tightness property in the sense of \cite[Section~6.4]{FOT}. By the strong Feller property of $X^n$ and \cite[Lemma~6.4.5]{FOT}, $\varphi_1$ can be taken to be positive on $K_n$.
\end{proof}
The heat kernel of $X^n$ is positive and continuous.
\begin{lemma}\label{lem:eigen}
For any $n \in \N$, there exists a continuous function $p_{t}^{n}(x,y):(0,\infty) \times K_n \times K_n \to (0,\infty)$ such that $$p_{t}^{n}f(x)=\int_{K_n}p_{t}^{n}(x,y)f(y)\, dm(y)$$ for any $t>0$, $x \in K_n$, and $f \in \mathcal{B}_{b}(K_n)$. 
\end{lemma}
\begin{proof}
Recall that $\{\lambda_k\}_{k=1}^{\infty} \subset [0,\infty]$ and $\{\varphi_k\}_{n=1}^{\infty}$ are the eigenvalues and the eigenfunctions of $-\mathcal{L}^n$, respectively. By Lemma~\ref{lem:contiversion}, we may assume $\{\varphi_k\}_{n=1}^{\infty}$ are bounded continuous on $K_n$. By \cite[Theorem~2.1.4]{D}, the series
\begin{equation*}
p_{t}^{n}(x,y):=\sum_{k=1}^{\infty}e^{-\lambda_{k} t} \varphi_{k}(x)\varphi_{k}(y)
\end{equation*}
absolutely converges uniformly on $[\eps,\infty) \times K_n \times K_n$ for any $\eps>0$. Each $\varphi_k$ is bounded continuous on $K_n$. Therefore, $p_{t}^n(x,y)$ becomes a bounded continuous function on $[\eps,\infty) \times K_n \times K_n$ for any $\eps>0$. For any $t>0$ and $f \in \mathcal{B}_{b}(K_n)$, $p_{t}^n(x,y)$ also defines an integral kernel of $\{p_{t}^n\}_{t>0}$: 
\begin{equation}
p_{t}^{n}f(x)=\int_{K_n}p_{t}^{n}(x,y)f(y)\,dm(y),\quad m \text{-a.e. }x\in K_n \label{eq:ahke}.
\end{equation}
By the positivity of $p_{t}^n$ and \eqref{eq:ahke}, $p_{t}^n(x,y) \ge 0$ for any $t>0$ and $(x,y) \in \ob{D} \times \ob{D}$. 
By Lemma~\ref{lem:ind}, $p_t^{n}f$ is a continuous function on $K_n$, and $p_{t}^n(x,y)$ is bounded continuous on $K_n \times K_n$, which implies that \eqref{eq:ahke} holds for any $x \in K_n$. 

Following the same argument as in \cite[Theorem~A.4]{Ki}, we prove the positivity of $p_{t}^n(x,y)$. By Lemma~\ref{lem:contiversion}, it holds that $\varphi_1(x)>0$ for any $x \in K_n$.
Therefore, it holds that
\begin{equation}
p_{t}^{n}(x,x)=\sum_{k=1}^{\infty}e^{-\lambda_k t}\varphi_k(x)^2>0 \label{eq:positive1}
\end{equation}
for any $t>0$ and $x \in K_n$. Let $x,y \in K_n$ and assume that $p_{s}(z,y)>0$ for some $s>0$. Then, for any $t,s >0$ with $t>s$, we have
\begin{equation}
p_{t}^n(x,y)=\int_{K_n}p_{s}^n(x,z)p_{t-s}^n(z,y)\,dm(z). \label{eq:positive2}
\end{equation}
Thus, $p_{t}^n(x,y)>0$ by the continuity of $p_{t}^n(x,y)$ and \eqref{eq:positive1}. This implies that there exists $t_\ast \in [0,\infty]$ such that $p_{t}^n(x,y)=0$ for any $t \in (0,t_{\ast}]$ and $p_{t}^n(x,y)>0$ for any $t \in (t_{\ast},\infty)$. We shall show that $t_{\ast}$ is finite. Since $K_n$ is arcwise connected, there exists a continuous function $\gamma:[0,1] \to K_n$ such that $\gamma(0)=x$ and $\gamma(1)=y$. By the continuity of $p_{t}^n(x,y)$ and \eqref{eq:positive1}, for any $s \in [0,1]$, there exists an open neighborhood $O_s \subset K_n$  of $\gamma(s)$ such that for any $z,w \in O_s$
\begin{equation*}
p_{1}^n(z,w)>0
\end{equation*}
Since $\gamma[0,1]$ is a compact subset of $K_n$, there exists $N \in \N$ and $\{s_i\}_{i=0}^{N}$ such that $0=s_0<s_1<\cdots<s_N=1$ and $x_i \in O_{i+1}$ for any $i=0,1,\ldots,N-1$, where $x_i=\gamma(s_i)$. \eqref{eq:positive1} and \eqref{eq:positive2} yield that
\begin{equation*}
p_{n'}^n(x,y)=\int_{K_n}\cdots \int_{K_n}p_{1}^n(x,x_1)p_{1}^n(x_1,x_2)\cdots p_{1}^n(x_{N-1},y)\,dm(x_1)\cdots dm(x_{N-1})>0,
\end{equation*}
which implies $t_{\ast} \le N<\infty$. Let $\mathbb{H}$ be the upper half-plane of $\mathbb{C}$. Then, 
$$\sum_{k=1}^{\infty}e^{-\lambda_{k} z} \phi_{k}(x)\phi_{k}(y)$$
converges uniformly on compact subsets of $\mathbb{H}$. Thus, $p_{z}^n(x,y)$ is extended to a holomorphic function on $\mathbb{H}$. If $t_{\ast}>0$, $p_{t}^n(x,y)=0$ for any $t \in (0,t_{\ast}]$. It also holds that $p_{z}^n(x,y)=0$ for any $z \in \mathbb{H}$. This contradicts to the fact that $t_{\ast}<\infty$. Hence, we have $t_{\ast}=0$.
\end{proof}

In what follows, $p_{t}^{n}(\cdot,\cdot)$ is extended to a function on $\ob {D} \times \ob{D}$ by setting $
p_{t}^n(\cdot,\cdot)=0 $ outside $K_n \times K_n.$
\begin{lemma}\label{lem:localunif}
It holds that
$
\lim_{n \to \infty}P_{x}(\tau_n \le t)=0
$ uniformly in $(t,x)$ over each compact subset of $[0,\infty) \times \ob{D}$.
\end{lemma}
\begin{proof}
By monotonicity, it suffices to show that $\lim_{n \to \infty}\sup_{x \in K}P_{x}(\tau_n \le t)=0$ for any $t>0$ and any
compact subset $K \subset \ob{D}$. We may assume $K \subset K_1$. It holds that
$$
P_{x}(\tau_n \le t)=1-P_{x}(t<\tau_n)=1-p_{t}^n\bone_{K_n}(x)
$$
for any $x \in K$ and $n \in \N$. By Lemma~\ref{lem:ind}, $p_{t}^{n}\bone_{K_n}$ is continuous on $K_n$. Hence, $P_{(\cdot)}(\tau_n \le t)$ is a continuous function on $K$. It follows from \cite[Lemma~6.8]{Ma} that
$$\lim_{m \to \infty}P_{x}(\tau_{m} \le t)=0$$
 for any $x \in K$. Since the convergence is monotone and non-increasing, we complete the proof by Dini's theorem.
\end{proof}

\section{Proof of Theorem~\ref{thm:thm1}}
In what follows, we fix $R \in (0,\infty)$. Recall that $\ob{D}_R$ is an open subset of $\ob{D}$: $\ob{D}_R=\ob{D} \cap B(R)$. We take $N\in \N$ such that
$
\bigcup_{(x,r) \in \ob{D}_R \times (0,R/2)}B(x,r) \subset K_N \subset J_{N+1}.
$ Note that $N$ depends only on $\ob{D}$ and $R$.
\begin{lemma}\label{lem:exit1} 
There exists a constant $\dl_R \in (0,1)$ which depends on $R$  such that
\begin{equation*}
Q_{x}^{N+1}(Y_{t}^{N+1} \in J_{N+1}\setminus B(x,r)) \le 1/4
\end{equation*}
for any $x \in \ob{D}_{R}$, $r \in (0,R/2)$, and $t  \in (0,\dl_R r^2]$.
\end{lemma}
\begin{proof}
Let $\dl_R \in (0,1)$ be a constant to be determined. 
By \eqref{nhke}, it holds that for any $(x,r) \in \ob{D}_{R} \times (0,R/2)$ and any $t  \in (0,\dl_R r^2]$
\begin{align*}
&Q_{x}^{N+1}(Y_{t}^{N+1} \in J_{N+1}\setminus B(x,r))\\
&\le a_{N+1}e^{t}t^{-d/2} \int_{\R^d \setminus B(r)} \exp\left( -|y|^2/(b_{N+1}t) \right)\,dy\\
&= \frac{2 \pi ^{d/2}a_{N+1}e^{t}t^{-d/2}}{\Gamma(d/2)}  \int_{r}^{\infty} s^{d-1} \exp(-s^2/b_{N+1}t)\,ds\\
&\le \frac{2 a_{N+1}  (\pi b_{n+1})^{d/2}e^{\dl_R R^2/4}}{\Gamma(d/2)}
 \int_{r/\sqrt{b_{N+1}t}}^{\infty} s^{d-1}\exp(-s^2)\,ds.
\end{align*}
Here $\Gamma$ is the gamma function. We take  $\dl_R \in (0,1 \wg (4/R^2))$ so that
\begin{equation*}
 \frac{2 a_{N+1}  (\pi b_{N+1})^{d/2}}{\Gamma(d/2)} \int_{1/\sqrt{b_{N+1}\dl_R}}^{\infty} s^{d-1}\exp(-s^2)\,ds \le \frac{1}{4},
 \end{equation*}
 which completes the proof.
\end{proof}
For each $(x,r) \in \ob{D}_{R} \times (0,R/2)$, we define stopping times as follows:
\begin{align*}
\tau_{B(x,r)}&=\inf \{t>0 \mid X_{t} \in \ob{D} \setminus B(x,r)\},\\
\tau_{B(x,r)}^N&=\inf \{t>0 \mid X_{t}^N \in \ob{D} \setminus B(x,r)\}, \\
T_{B(x,r)}^{N+1}&=\inf \{t>0 \mid Y^{N+1}_{t} \in J_{N+1} \setminus B(x,r)\}, \\
T_{B(x,r)}^{N+1,N}&=\inf \{t>0 \mid Y^{N+1,N}_{t} \in K_{N+1} \setminus B(x,r)\}.
\end{align*}
Using Lemma~\ref{lem:exit1} and applying \cite[Theorem~7.2]{GK} to the conservative diffusion process $Y^{N+1}$ on $J_{N+1}$, we obtain the next corollary.
\begin{cor}\label{cor:exit1}
There exist positive constants $c_R,\gamma_R$ depend on $R$ such that 
\begin{equation*}
Q_{x}^{N+1}(T_{B(x,r)}^{N+1} \le t) \le c_R \exp (-\gamma_R r^2/t)
\end{equation*}
for any $(x,r) \in \ob{D}_R \times (0,R)$ with $B(x,r) \subset B(R)$ and $t \in (0,\infty)$.
\end{cor}

We shall give a proof of Theorem~\ref{thm:thm1}.
\begin{proof}[Proof of Theorem~\ref{thm:thm1}~(i)]
Let $t \in (0,\infty)$ and $(x,r) \in \ob{D}_R \times (0,R)$ with $B(x,r) \subset B(R)$. It clearly holds that $T^{N+1,N}_{B(x,r)}=T^{N}_{B(x,r)}$. Hence, it follows that
\begin{equation}
Q_{x}^{N+1}(T_{B(x,r)}^{N+1,N} \le t)=Q_{x}^{N+1}(T_{B(x,r)}^{N+1} \le t)\label{eq:eq3}.
\end{equation}
By Lemma~\ref{lem:ind}, we have
\begin{equation}
Q_{x}^{N+1}(T_{B(x,r)}^{N+1,N} \le t)=P_{x}(\tau_{B(x,r)}^{N} \le t) \label{eq:eq2}.
\end{equation}
Since $\tau_{B(x,r)}^{N}=\tau_{B(x,r)}$, it holds that
\begin{equation}
P_{x}(\tau_{B(x,r)}^N \le t)=P_{x}(\tau_{B(x,r)} \le t).\label{eq:eq1}
\end{equation}
By using the equalities \eqref{eq:eq1}, \eqref{eq:eq2}, \eqref{eq:eq3}, and Corollary~\ref{cor:exit1}, we complete the proof. 
\end{proof}
\begin{proof}[Proof of Theorem~\ref{thm:thm1}~(ii)]
We check the conditions $\text{(DB)}_{\beta}$, $\text{(DU)}_{F}^{U,R}$, and $\text{(P)}_{\beta}^{U,R}$ in \cite[Theorem~1.1]{GK} with $\beta=2$, $U=\ob{D}_R$. The third condition has already shown in Theorem~\ref{thm:thm1}~(i). We define $F_{t}(x,y): (0,R^2] \times\ob{D}_{R} \times \ob{D}_{R} \to [0,\infty)$ by $F_{t}(x,y)=a_{N+1} e^{t}t^{-d/2}.$ Then,  it holds that $F_{s}(z,w)/F_{t}(x,y)\le (t/s)^{d/2}$
 for any $(t,x,y), (s,z,w) \in  (0,R^2] \times \ob{D}_{R} \times \ob{D}_{R}$ with $s\le t$. This implies $\text{(DB)}_{\beta}$. We shall check $\text{(DU)}_{F}^{U,R}$. 
It follows from Lemma~\ref{lem:ind} and  \eqref{nhke} that
\begin{align*}
&P_{x}(X_t \in A, t <\tau_{ \ob{D}_R}) =P_{x}(X_t^N \in A, t <\tau_{ \ob{D}_R}) \\
&\le Q_{x}(Y_t^{N+1,N} \in A)\le \int_{A}a_{N+1} e^{t} t^{-d/2}\exp\left(-|x-y|^2/b_{N+1} t \right)\,dm(y) \\
&\le \int_{A}F_{t}(x,y)\,dm(y)
\end{align*}
for any $(t,x) \in (0,R^2) \times \ob{D}_R$ and any $A \in \mathcal{B}(\ob{D}_{R})$, which implies $\text{(DU)}_{F}^{U,R}$.
\end{proof}
\begin{proof}[Proof of Theorem~\ref{thm:thm1}~(iii)]
We write $E_x$ for the expectation with respect to the probability measure $P_x$.
Take $\eps \in (0,1)$ and $R>0$.  Let $f \in \mathcal{B}_{b}(\ob{D})$ be a nonnegative function with $f|_{\ob{D} \setminus \ob{D}_{\eps,R}}=0$. Let $n,n' \in \N$ with $n>n'$ and $\ob{D}_{R+1} \subset K_{n'}$. Recall $\tau_n=\inf \{t>0 \mid X_{t} \in \ob{D} \setminus K_n\}$. It holds that $P_{x}(X_{\tau_{n'}} \in K_{n'})=0$ for any $x \in K_{n'}$. For any $x \in K_{n'}$ and any $t>0$,  we have $E_{x}[\bone_{\{\tau_{n'}=t\}}f(X_t)]=0$ since $f=0$ on $\ob{D} \setminus K_{n'}$. Thus, it holds that
\begin{align}
p_{t}^{n}f(x)
&=E_{x}[f(X_t):t<\tau_n]=p_{t}^{n'}f(x)+E_{x}[f(X_t) \bone_{\{ \tau_{n'}<t<\tau_n \}}]. \label{eq:eql}
\end{align}
for any $t>0$ and $x \in K_{n'}$. We denote by $\{\theta_t\}_{t \ge 0}$ the shift operator of $X$. Using the relation $\tau_{n'} \le \tau_n =\tau_{n'}+\tau_{n} \circ \theta_{\tau_{n'}}$ and the strong Markov property \cite[Proposition~3.4]{GK} of $X$, we obtain
\begin{align}
&E_{x}[f(X_t) \bone_{\{ \tau_{n'}<t<\tau_n \}}] \notag \\
&=E_{x}[\bone_{\{ \tau_{n'}<t\}} \bone_{\{ t< \tau_{n'}+\tau_{n} \circ \theta_{\tau_{n'}}\}} f(X_t) ] \notag \\
&=E_{x}[\bone_{\{ \tau_{n'}<t\}}E_{X_{\tau_{n'}}}[f(X_{t-\tau_{n'}})\bone_{\{ t- \tau_{n'}<\tau_{n} \}}] ] \notag \\
&=E_{x}[\bone_{\{ \tau_{n'}<t\}}E_{X_{\tau_{n'}}}[f(X_{t-\tau_{n'}}^n)] ]=E_{x}[\bone_{\{ \tau_{n'}<t\}}p_{t-\tau_{n'}}^{n}f(X_{\tau_{n'}}) ] \label{eq:eql2}.
\end{align}
It follows from \eqref{eq:eql} and \eqref{eq:eql2} that
\begin{align}
0 &\le p_{t}^{n}f(x)-p_{t}^{n'}f(x)=E_{x}[\bone_{\{ \tau_{n'}<t\}}p_{t-\tau_{n'}}^{n}f(X_{\tau_{n'}}) ] \notag \\
&=  E_{x}\left[\int_{\ob{D}_{\eps,R}}p_{t-\tau_{n'}}^{n}(X_{\tau_{n'}},y)f(y)\,dm(y):\tau_{n'}<t\right]. \label{inq:inqbound3}
\end{align}
It is easy to see
\begin{align}
&\sup_{s \in (0, R^2 \wg t]}\sup_{x \in J_m \setminus \ob{D}_{R+1}} \sup_{y \in \ob{D}_{\eps,R}}p_{s}^{n}(x,y) \le \sup_{s \in (0,R^2 \wg t]}\sup_{x \notin \ob{D}_R} \esssup_{y \in \ob{D}_{\eps,R}}p_{s}(x,y) \label{eq:inqbound1},
\end{align}
where $\esssup$ denotes the essential supremum with respect to $m$.
It also holds that
\begin{align}
&\sup_{s \in [R^2 \wg t, t]}\sup_{x \in J_m \setminus \ob{D}_{R+1}}  \sup_{y \in \ob{D}_{\eps,R}}p_{s}^{n}(x,y)  \le \sup_{s \in [R^2 \wg t, t]}\sup_{x \in \ob{D}} \esssup_{y \in \ob{D}_{\eps,R}}p_{s}(x,y). \label{eq:inqbound2}
\end{align}
By Theorem~\ref{thm:thm1}~(i), both \eqref{eq:inqbound1} and \eqref{eq:inqbound2} are bounded above by a positive constant depends on $\eps$ and $R$, say $C_{\eps,R}$. Since $X_{\tau_{n'}} \in J_{n'} \setminus \ob{D}_{R+1}$, \eqref{inq:inqbound3} implies that 
\begin{equation}
0 \le p_{t}^{n}(x,y)-p_{t}^{n'}(x,y)\le C_{\eps,R} \times P_{x}[\tau_{n'} \le t] \label{eq:conti}
\end{equation}
for any $(t,x,y) \in (0,\infty) \times K_{n'} \times \ob{D}_{\eps,R}$. 

For each $(t,x,y) \in (0,\infty) \times \ob{D} \times \ob{D}$, we define 
\begin{equation*}
p_{t}^{\ast}(x,y):=\lim_{n \to \infty}p_{t}^{n}(x,y).
\end{equation*}
By \eqref{eq:conti}, Lemma~\ref{lem:eigen}, and Lemma~\ref{lem:localunif}, $p_{t}^{\ast}(x,y)$ is a continuous function on $ (0,\infty) \times \ob{D} \times \ob{D}_{\eps,R}$. Since $\eps \in (0,1)$, $R>0$ are arbitrarily chosen, $p_{t}^{\ast}(x,y)$ is also  continuous on $(0,\infty)\times \ob{D} \times \ob{D}$. For any $x \in \ob{D}$, $t>0$, and nonnegative function $f \in \mathcal{B}_{b}(\ob{D})$, we have
\begin{align*}
E_{x}[f(X_t):t<\tau_n]&=p_{t}^{n}f(x)=\int_{\ob{D}}p_{t}^{n}(x,y)f(y)\,dm(y).
\end{align*}
Monotone convergence theorem and Lemma~\ref{lem:localunif} yield
\begin{equation*}
E_{x}[f(X_t)]=\int_{\ob{D}}p_{t}^{\ast}(x,y)f(y)\,dm(y).
\end{equation*}
 By Lemma~\ref{lem:eigen}, $p_{t}^n(x,y)>0$ for any $t>0$, $n \in \N$, and $x,y \in K_n$. Hence, we have $p_{t}^{\ast}(x,y)>0$ for any $t>0$ and $x,y \in \ob{D}$. 
\end{proof}

\section{Proof of Theorem~\ref{cor:2}}
In what follows, we fix a non-empty open subset $U \subset \ob{D}$. Recall that each $X^{U \cap K_n}$ is the part process of $X$ on the open subset $U \cap K_n$ of $ \ob{D}$. Each $X^{U \cap K_n}$ is also regarded as the part process of $X^n$ on $U \cap K_n$. Thus, by Lemma~\ref{lem:ind} and \cite[Theorem~1]{CK}, the semigroup $\{p_{t}^{U, n}\}_{t >0}$ of $X^{U \cap K_n}$ is strong Feller. By repeating the same arguments as in Lemma~\ref{lem:contiversion} and Lemma~\ref{lem:eigen}, we obtain the next lemma.
\begin{lemma}\label{lem:localconti}
For any $n \in \N$, there exists a continuous function $p_{t}^{U,n}(x,y):(0,\infty) \times (U \cap K_n) \times (U \cap K_n) \to (0,\infty)$ such that $$p_{t}^{U,n}f(x)=\int_{K_n}p_{t}^{U,n}(x,y)f(y)\, dm(y)$$ for any $t>0$, $x \in U \cap K_n$ and $f \in \mathcal{B}_{b}(U \cap K_n)$. 
\end{lemma}
$p_{t}^{U,n}(\cdot,\cdot)$ is extended to a function on $\ob{D} \times \ob{D}$ by setting $p_{t}^{U,n}(\cdot,\cdot)=0$ outside $(U \cap K_n) \times (U \cap K_n) $. 
\begin{proof}[Proof of Corollary~\ref{cor:2}]
If $U$ is bounded, we complete the proof by Lemma~\ref{lem:localconti}. Therefore, we may assume that $U$ is unbounded. Let $n,n' \in \N$ with $n>n'$. Then, 
\begin{equation}
0 \le p_{t}^{U,n}(x,y)-p_{t}^{U,n'}(x,y) \le p_{t}^{n}(x,y)-p_{t}^{n'}(x,y),\quad t>0,\ x,y\in \ob{D}. \label{repeat}
\end{equation}
To see the latter inequality, note that for $(x,y) \in ( \ob{D} \times \ob{D}) \setminus ((U \cap K_{n'})\times (U \cap K_{n'}))$ this inequality holds trivially. For $(x,y) \in (U \cap K_{n'})\times (U \cap K_{n'})$, it holds that
\begin{align*}
 &p_{t}^{n}(x,y)-p_{t}^{n'}(x,y)-p_{t}^{U,n}(x,y)+p_{t}^{U,n'}(x,y)  \\
&=\lim_{r \to 0}\frac{P_{x}(X_t \in \ob{D} \cap B(y,r),\ \tau_{U } \vee \tau_{n'} \le t<\tau_n)}{m(\ob{D} \cap B(y,r))} \ge 0
\end{align*}
by the continuity of the densities of $X^{\ob{D} \cap B(y,r)}$, $r>0$, which is assured by Lemma~\ref{lem:localconti}. 
By using  \eqref{repeat} and repeating the same argument as in the proof of Theorem~\ref{thm:thm1}~(iii), we complete the proof. 
\end{proof}

\section{Application}
For each open subset $O \subset \ob{D}$, we define $\text{Cap}_{\ob{D}}(O)$ by 
\begin{align*}
\text{Cap}_{\ob{D}}(O)=\inf \{\|f\|_{H^{1}(D)}^2 \mid f \in H^{1}(D),\ f \ge 1,\ m\text{-a.e. on }O \}.
\end{align*}
For each subset $A \subset \ob{D}$, we set $$\text{Cap}_{\ob{D}}(A)=\inf \{\text{Cap}_{\ob{D}}(O) \mid A \subset O,\ O \subset \ob{D}\text{ is an open subset}\}.$$
Let $\mathcal{H}^{d-1}$ be the $(d-1)$-dimensional Hausdorff measure on $\R^d$. We denote by $\sigma$ the restriction on $\partial D:=\ob{D} \setminus D$. $\sigma$ is a Radon measure on $(\partial D, \mathcal{B}(\partial D))$. It is shown in \cite[Proposition~2.4]{Ma} that $\sigma$ is also a smooth measure: $\sigma(A)=0$ whenever $\text{Cap}_{\ob{D}}(A)=0$, $A \subset \partial D$. By \cite[Theorem~5.1.3]{FOT}, there is a unique positive continuous additive functional $L=\{L_t\}_{t \ge 0}$ of $X$ such that 
\begin{equation*}
\int_{\ob{D}}h(x)E_{x}\left[\int_{0}^{t}f(X_s)\,dL_s \right]\,dm(x)=\int_{0}^{t}\int_{\partial D}f(x)(p_{s}h)(x)\,d\sigma(x)\,ds
\end{equation*}
for any $t>0$ and $f,h \in \mathcal{B}_{b}(\ob{D})$. We call $L$ the {\it boundary local time} of $X$. By \eqref{ac} and the Markov property of $X$, it holds that
\begin{equation}
E_{x}\left[\int_{0}^{t}f(X_s)\,dL_s \right]=\int_{0}^{t}\int_{\partial D}p_{s}(x,y)f(y)\,d\sigma(y)\,ds\label{eq:absigma}
\end{equation}
for any $t>0,\, x \in \ob{D}$, and $f\in \mathcal{B}_{b}(\ob{D})$.

$\sigma$ is in the local Kato class of $X$ in the sense of \cite{CK}.
\begin{theorem}\label{localkato}
It holds that 
$$\lim_{t \to 0}\sup_{x \in \ob{D}}E_{x}\left[\int_{0}^{t}\bone_{K}(X_s)\,dL_s \right]=0$$
for any relatively compact open subset $K$ of $\ob{D}$.
\end{theorem}
\begin{proof}
Since $D$ is a Lipschitz domain, $\partial D \cap K$ is a part of the boundary of a bounded Lipschitz domain $E$ of $\R^d$. It follows from \eqref{eq:absigma} that for any $t>0$
\begin{align*}
&\sup_{x \in \ob{D}}E_{x}\left[\int_{0}^{t \wg 1}\bone_{K}(X_s)\,dL_s \right] \\
&\le \int_{0}^{t \wg 1}\sup_{x \in \ob{E}}\int_{\partial D \cap K}p_{s}(x,y)\,d\mathcal{H}^{d-1}(y)\,ds+ \int_{0}^{t \wg 1}\sup_{x \in \ob{D}\setminus \ob{E}}\int_{\partial D \cap K}p_{s}(x,y)\,d\mathcal{H}^{d-1}(y)\,ds \\
&=:I_1+I_2.
\end{align*}
By Theorem~\ref{thm:thm1}~(i) and (iii), there exist $c_1, c_2 \in (0,\infty)$ depending on $E$ such that for any $t>0$
\begin{align}
I_1&\le c_1\times \sup_{x \in \ob{E}}\int_{0}^{t \wg 1}s^{-d/2}\int_{\partial E}\exp(- c_2|x-y|^2/s )\,d\mathcal{H}^{d-1}(y)\,ds, \label{eq:I1} \\
 I_2& \le c_1\times \mathcal{H}^{d-1}(\partial E) \times \int_{0}^{t \wg 1}  (2s)^{-d/2}  \exp(-c_2/s )\,ds. \notag
\end{align}
It is easy to see $\lim_{t \to 0}I_{2}=0.$ Thus, it remains to show $\lim_{t \to 0}I_1=0$. For $\eps>0$, we define $E_{\eps}=\{x \in E \mid \text{dist}(x,\partial E)<\eps\}$. As $E$ is a bounded Lipschitz domain, there exist $\eps_{0}>0$ and $c_3=c_3(d,E,c_2)>0$ such that 
\begin{equation}
\frac{1}{\eps}\int_{E_{\eps}}s^{-d/2}\exp (- c_2 |x-y|^2/s )\,dm(y) \le c_3/\sqrt{s} \label{eq:I11}
\end{equation}
for any $s \in (0,1],$ any $\eps \in (0,\eps_0)$ and any $x \in \ob{E}$. See \cite[Lemma~7,4]{Ma} for the proof. Using \eqref{eq:I1}, \eqref{eq:I11} and \cite[Lemma~7.1]{CF}, we obtain
\begin{align*}
\varlimsup_{t \to 0}I_1&\le c_1  \times \varlimsup_{t \to 0}  \sup_{x \in E} \varliminf_{\eps \to 0} \int_{0}^{t \wg 1}\frac{1}{\eps}\int_{E_{\eps}}s^{-d/2}\exp (- c_2 |x-y|^2/s )\,dm(y)\,ds \\
&\le c_{1}c_3 \times \varlimsup_{t \to 0}\int_{0}^{t \wg 1} s^{-1/2}\,ds=2c_{1}c_{3}\times \varlimsup_{t \to 0}\sqrt{t}=0,
\end{align*}
which completes the proof.
\end{proof}

\begin{remark}
If $L$ satisfies $\lim_{t \to 0}\sup_{x \in \ob{D}}E_{x}[L_t]=0$, $\sigma$ is said to be in the Kato class of $X$. 
If $D$ is thin at infinity: $\lim_{|x| \to \infty,\ x \in \ob{D}}m(D \cap B(x,1))=0$, it is shown in the proof of \cite[Corollary~2.8]{Ma} that $\lim_{|x| \to \infty,\ x \in \ob{D}}E_{x}[\exp(-L_t)]=0$  for any $t>0$. It follows from Jensen's inequality that $\sup_{x \in \ob{D}}E_{x}[L_t] \ge \lim_{|x| \to \infty,\ x \in \ob{D}}E_{x}[L_t]=\infty$ for any $t>0$. Thus, $\sigma$ is generally not in the Kato class of $X$. 
\end{remark}

\noindent
{\it Acknowledgement}
The author would like to thank professor Naotaka Kajino for his helpful comments on the proof of Lemma~\ref{lem:ind}.


\begin{bibdiv}
\begin{biblist}
\bib{BH}{article}{
   author={Bass, Richard F.},
   author={Hsu, Pei},
   title={Some potential theory for reflecting Brownian motion in H\"older and
   Lipschitz domains},
   journal={Ann. Probab.},
   volume={19},
   date={1991},
   number={2},
   pages={486--508}
}

\bib{BCM}{article}{
   author={Burdzy, Krzysztof},
   author={Chen, Zhen-Qing},
   author={Marshall, Donald E.},
   title={Traps for reflected Brownian motion},
   journal={Math. Z.},
   volume={252},
   date={2006},
   number={1},
   pages={103--132}
}

\bib{CK}{article}{
   author={Chen, Zhen-Qing},
   author={Kuwae, Kazuhiro},
   title={On doubly Feller property},
   journal={Osaka J. Math.},
   volume={46},
   date={2009},
   number={4},
   pages={909--930}
}

\bib{CF}{article}{
   author={Chen, Zhen-Qing},
   author={Fan, Wai-Tong},
   title={Systems of interacting diffusions with partial annihilation
   through membranes},
   journal={Ann. Probab.},
   volume={45},
   date={2017},
   number={1},
   pages={100--146},
   issn={0091-1798},
}
\bib{D}{book}{
   author={Davies, E. B.},
   title={Heat kernels and spectral theory},
   series={Cambridge Tracts in Mathematics},
   volume={92},
   publisher={Cambridge University Press, Cambridge},
   date={1990},
   pages={x+197}
}

\bib{EE}{book}{
   author={Edmunds, D. E.},
   author={Evans, W. D.},
   title={Spectral theory and differential operators},
   series={Oxford Mathematical Monographs},
   note={Oxford Science Publications},
   publisher={The Clarendon Press, Oxford University Press, New York},
   date={1987},
   pages={xviii+574},
   isbn={0-19-853542-2},
}

\bib{FOT}{book}{
   author={Fukushima, Masatoshi},
   author={Oshima, Yoichi},
   author={Takeda, Masayoshi},
   title={Dirichlet forms and symmetric Markov processes},
   series={De Gruyter Studies in Mathematics},
   volume={19},
   edition={Second revised and extended edition},
   publisher={Walter de Gruyter \& Co., Berlin},
   date={2011},
   pages={x+489},
   isbn={978-3-11-021808-4}
}

\bib{FT}{article}{
   author={Fukushima, Masatoshi},
   author={Tomisaki, Matsuyo},
   title={Construction and decomposition of reflecting diffusions on
   Lipschitz domains with H\"{o}lder cusps},
   journal={Probab. Theory Related Fields},
   volume={106},
   date={1996},
   number={4},
   pages={521--557},
   issn={0178-8051}
   }
   

\bib{GK}{article}{
   author={Grigor'yan, Alexander},
   author={Kajino, Naotaka},
   title={Localized upper bounds of heat kernels for diffusions via a
   multiple Dynkin-Hunt formula},
   journal={Trans. Amer. Math. Soc.},
   volume={369},
   date={2017},
   number={2},
   pages={1025--1060},
}

   \bib{GS}{article}{
   author={Gyrya, Pavel},
   author={Saloff-Coste, Laurent},
   title={Neumann and Dirichlet heat kernels in inner uniform domains},
   language={English, with English and French summaries},
   journal={Ast\'{e}risque},
   number={336},
   date={2011},
   pages={viii+144},
   issn={0303-1179}
}



\bib{Ki}{article}{
   author={Kigami, Jun},
   title={Volume doubling measures and heat kernel estimates on self-similar
   sets},
   journal={Mem. Amer. Math. Soc.},
   volume={199},
   date={2009},
   number={932},
   pages={viii+94}
}

\bib{KKT}{article}{
   author={Kurniawaty, Mila},
   author={Kuwae, Kazuhiro},
   author={Tsuchida, Kaneharu},
   title={On the doubly Feller property of resolvent},
   journal={Kyoto J. Math.},
   volume={57},
   date={2017},
   number={3},
   pages={637--654}
}
\bib{Ma}{article}{
   author={Matsuura, Kouhei},
   title={Doubly Feller property of Brownian motions with Robin boundary condition},
   journal={To appear in Potential Anal. Available at \url{https://link.springer.com/article/10.1007/s11118-018-09758-4}}
}
\bib{Ma2}{article}{
   author={Matsuura, Kouhei},
   title={$L^p$-spectral independence of Neumann laplacians on horn-shaped domains},
   journal={preprint}
}
\bib{O}{book}{
   author={Ouhabaz, El Maati},
   title={Analysis of heat equations on domains},
   series={London Mathematical Society Monographs Series},
   volume={31},
   publisher={Princeton University Press, Princeton, NJ},
   date={2005},
   pages={xiv+284},
}
\end{biblist}
\end{bibdiv}
\end{document}